\documentclass[12pt]{article}

\usepackage{amsfonts}
\usepackage{graphicx}
\usepackage{amsmath}
\usepackage{amssymb}
\usepackage[numbers,square]{natbib}
\usepackage{url}
\usepackage{multirow}
\usepackage{color}
\usepackage{amsthm}

\newtheorem{thm}{Theorem}
\newtheorem{cor}[thm]{Corollary}

\newtheorem{lem}[thm]{Lemma}

\newtheorem{remark}[thm]{Remark}

\numberwithin{thm}{section}
\numberwithin{equation}{section}

\title{Compound geometric approximation under a failure rate constraint}
\author{Fraser Daly\footnote{Department of Actuarial Mathematics and Statistics, Heriot-Watt University, Edinburgh EH14 4AS, UK.  E-mail: f.daly@hw.ac.uk;  Tel: +44 (0)131 451 3212; Fax: +44 (0)131 451 3249}}
\date{}

\begin{document}

\maketitle

\noindent{\bf Abstract} 
We consider compound geometric approximation for a nonnegative, integer-valued random variable $W$.  The bound we give is straightforward but relies on having a lower bound on the failure rate of $W$. Applications are presented to M/G/1 queuing systems, for which we state explicit bounds in approximations for the number of customers in the system and the number of customers served during a busy period.  Other applications are given to birth-death processes and Poisson processes.
\vspace{12pt}


\vspace{12pt}


\section{Introduction and main result}\label{sec:intro}

We consider the approximation of a nonnegative, integer-valued random variable $W$ by a compound geometric distribution.  We say that $Y$ has a compound geometric distribution if it is equal in distribution to $\sum_{i=1}^NX_i$, where $X,X_1,X_2,\ldots$ are i$.$i$.$d$.$ and $N\sim\mbox{Geom}(p)$ has a geometric distribution with $\mathbb{P}(N=k)=p(1-p)^k$ for $k=0,1,2,\ldots$.  In this work we will only consider the case where $X$ takes values in $\mathbb{N}=\{1,2,\ldots\}$.  As usual, the empty sum is treated as zero, so that $\mathbb{P}(Y=0)=\mathbb{P}(N=0)=p$.

Such compound geometric distributions arise in a number of applications in a variety of fields, including reliability, queueing theory and risk theory: see \cite{k97} for an overview.  It is well-known that a compound geometric distribution converges to an exponential distribution as $p\rightarrow0$.  Explicit bounds in exponential approximation for compound geometric distributions have been given by Brown \cite{b90,b15}, Bon \cite{b06}, and Pek\"oz and R\"ollin \cite{pr11}.  Brown's work takes advantage of reliability properties of such compound geometric distributions; such properties will also prove useful in our work.  The bounds given by Pek\"oz and R\"ollin \cite{pr11} apply more generally than to compound geometric distributions, relaxing the assumptions that $N$ have a geometric distribution and that the $X_i$ be independent.  Pek\"oz, R\"ollin and Ross \cite{prr13} give bounds in the geometric approximation of compound geometric distributions.  Note that some of the above-mentioned bounds apply in the case where $N$ is supported on $\{0,1,\ldots\}$, and some in the case where $N$ has support $\{1,2,\ldots\}$.

Here we will consider the approximation of $W$ by our compound geometric random variable $Y$ using the total variation distance, defined by
$$
d_{TV}(\mathcal{L}(W),\mathcal{L}(Y))=\sup_{A\subseteq\mathbb{Z}^+}\left|\mathbb{P}(W\in A)-\mathbb{P}(Y\in A)\right|\,,
$$
where $\mathbb{Z}^+=\{0,1,2,\ldots\}$.  Some work on compound geometric approximation in total variation distance has been done by Daly \cite{d10}, whose main application is to hitting times of Markov chains in a quite general setting.  We build upon that work by presenting bounds which are more straightforward to evaluate, but which require some knowledge about the behaviour of the failure rate of the random variable $W$.  
We recall that the failure rate (or hazard rate) of a nonnnegative, integer-valued random variable $W$ is defined to be
$$
r_W(j)=\frac{\mathbb{P}(W=j)}{\mathbb{P}(W>j)}\,,\qquad j\in\mathbb{Z}^+\,.
$$ 
Some authors use an alternative definition, taking the failure rate of $W$ to be
$$
\widetilde{r}_W(j)=\frac{\mathbb{P}(W=j)}{\mathbb{P}(W\geq j)}\,,\qquad j\in\mathbb{Z}^+\,.
$$
The failure rate of a continuous random variable may be defined analogously, by replacing the mass function in the numerators of the above with a density function. 

Note that our bounds may be applied in conjunction with bounds for exponential or geometric approximation of compound geometric distributions discussed above.

In our main result, Theorem \ref{thm:main} below, we will assume that we have $\delta>0$ such that $r_W(j)\geq\delta$ for all $j$.  Given such a $\delta$, the total variation distance between $W$ and a compound geometric distribution may be effectively bounded by computing their expectations.  This is in contrast to the bounds presented in \cite{d10}, where more detailed information must be known about $W$ to allow them to be computed.

We note the work by Brown and Kemperman \cite{bk09} and Brown \cite{b14}, who find bounds on the distribution function and variance of a random variable, respectively, under bounds on its failure rate.  Explicit bounds in probability approximation for a random variable with a bounded failure rate have also been derived in the recent work of Brown \cite{b15}.  He gives sharp bounds in exponential approximation for random variables whose failure rate may be bounded from above, with applications to compound geometric distributions, and first passage times of birth-death processes and other reversible Markov chains (in continuous time).  Note that here we are working with discrete random variables, under the assumption of a lower bound on the failure rate.  Our results complement, but do not overlap with, Brown's work. 

After stating our main theorem, a first application (approximating the equilibrium number of customers in an M/G/1 queueing system) will be presented to illustrate our bound.  Further applications will be given in Sections \ref{sec:ifr} and \ref{sec:dfr}.  

In Section \ref{sec:ifr} we will consider the well-studied problem of geometric approximation for random variables with increasing failure rate.  We will consider two straightforward applications of our Theorem \ref{thm:main} (to Poisson processes and the P\'olya distribution) which allow us to explicitly compare our bound with a similar result from \cite{o81}.  In Section \ref{sec:dfr} we consider compound geometric approximation for random variables with decreasing failure rate.  In particular, we consider the number of customers served during a busy period of an M/G/1 queue, and the time to extinction of a discrete birth-death process.     

The proof of our Theorem \ref{thm:main} is given in Section \ref{sec:proof}.  The proof uses Stein's method (see \cite{bc05} and references therein), building upon previous work on Stein's method for geometric \cite{p96} and compound geometric \cite{d10} approximation.  Finally, in Section \ref{sec:hr} we give related results which we illustrate with short examples.

For future use, for any $0\leq p\leq1$ and positive, integer-valued random variable $X$ with $u=1-d_{TV}(\mathcal{L}(X),\mathcal{L}(X+1))$, we define
\begin{equation}\label{eq:steinfactor}
H_p(X)=\min\left\{p+(1-p)\mathbb{P}(X>1),p\left(1+\sqrt{\frac{-2}{u\log(1-p)}}\right)\right\}\,.
\end{equation}
We now state our main result.
\begin{thm}\label{thm:main}
Let $W$ be a nonnegative, integer-valued random variable with $\mathbb{P}(W=0)=p\in(0,1)$ and $r_W(j)\geq\delta>0$ for all $j$.  Let $Y=\sum_{i=1}^NX_i$, where $N\sim\mbox{Geom}(p)$ and $X,X_1,X_2,\ldots$ are i$.$i$.$d$.$ positive, integer-valued random variables.  If
$$
\mathbb{E}X\geq\frac{p}{(1-p)\delta}\,,
$$
then
$$
d_{TV}(\mathcal{L}(W),\mathcal{L}(Y))\leq H_p(X)\left(\mathbb{E}Y-\mathbb{E}W\right)\,.
$$
\end{thm}
Note that under the conditions of this theorem, it is straightforward to show that $\mathbb{E}W\leq\delta^{-1}$, so that the resulting upper bound is nonnegative, as expected.

\subsection{Application to the number of customers in an M/G/1 queue}\label{sec:queue}

Consider an M/G/1 queueing system in equilibrium, with customers arriving at rate $\lambda$ and with i$.$i$.$d$.$ service times having the same distribution as the random variable $S$.  Letting $\rho=\lambda\mathbb{E}[S]$, we assume throughout that $\rho<1$.  
Let $W$ be the number of customers in the system.  It is well-known that
$$
\mathbb{P}(W=0)=1-\rho\,,\qquad\mbox{and}\qquad
\mathbb{E}W=\rho+\frac{\rho^2\mathbb{E}[S^2]}{2(1-\rho)(\mathbb{E}S)^2}\,.
$$ 
See, for example, page 281 of \cite{a03}.

Let $R_j$ denote the residual service time of the customer currently being served in the queue, conditional on the event $\{W=j\}$.  Then Ross \cite{r06} shows that
$$
r_W(j)=\frac{1-\rho}{\lambda\mathbb{E}R_j}\,,\qquad\mbox{ and }\qquad\mathbb{E}R_j\leq\sup_{t\in\mathbb{R}^+}\mathbb{E}\left[S-t|S\geq t\right]\,.
$$ 
We may thus apply Theorem \ref{thm:main} with the choice
$$
\delta=\frac{1-\rho}{\lambda\sup_{t\in\mathbb{R}^+}\mathbb{E}\left[S-t|S\geq t\right]}\,.
$$
The random variable $S$ is said to be new better than used in expectation (NBUE) if we have that $\mathbb{E}\left[S-t|S\geq t\right]\leq\mathbb{E}S$ for all $t\geq0$.  In this case we may take $\delta=\rho^{-1}(1-\rho)$ and, in the notation of Theorem \ref{thm:main}, we may then take $X=1$ a$.$s$.$, so that $Y$ simply has a geometric distribution and $H_p(X)=p$.  We thus obtain the following.
\begin{cor}\label{prop:q1}
Let $W$ be the number of customers in an M/G/1 queueing system in equilibrium as above.  If $S$ is NBUE then
$$
d_{TV}(\mathcal{L}(W),\mbox{Geom}(1-\rho))\leq\rho^2\left(1-\frac{\mathbb{E}[S^2]}{2(\mathbb{E}S)^2}\right)\,.
$$
\end{cor}
Note that, as expected, this upper bound is zero if the service time $S$ has an exponential distribution (which is indeed NBUE).

Finally in this section, we refer the interested reader to the book by M\"uller and Stoyan \cite{ms02}, who prove many stochastic comparison and monotonicity results for queueing models (and in many other applications), and derive associated bounds on quantities such as the mean waiting time and mean busy period for stationary queues.  Some of their work also takes advantage of reliability properties of the underlying random variables, as we have done here. 

\section{Geometric approximation for IFR distributions}\label{sec:ifr}

In the notation of Theorem \ref{thm:main}, since $p=\mathbb{P}(W=0)$ we have that $r_W(0)=p(1-p)^{-1}$.  If the failure rate $r_W(j)$ is increasing in $j$, this may clearly serve as the lower bound $\delta$.  In this case, we may let the random variable $X$ be 1 almost surely, so that $Y$ has a geometric distribution. Noting that $H_p(X)=p$ in this case, we obtain the following.
\begin{cor}\label{cor:ifr}
Let $W$ be a nonnegative, integer-valued random variable with $\mathbb{P}(W=0)=p$.  If $W$ has increasing failure rate (IFR) then
$$
d_{TV}(\mathcal{L}(W),\mbox{Geom}(p))\leq 1-p(1+\mathbb{E}W)\,.
$$
\end{cor}
Note that we do not need the monotonicity of $r_W$ to obtain such a bound; it suffices to have $r_W(j)\geq r_W(0)$ for all $j\in\mathbb{Z}^+$.
  
Geometric approximation theorems for IFR random variables are well-known.  We use the remainder of this section to give two explicit examples in which we can compare our Corollary \ref{cor:ifr} to the main theorem of Obretenov \cite{o81}.  Obretenov does not use total variation distance $d_{TV}$, but employs the Kolmogorov distance $d_K$ defined by
$$
d_K(\mathcal{L}(W),\mathcal{L}(Y))=\sup_{j\in\mathbb{Z}^+}\left|\mathbb{P}(W\leq j)-\mathbb{P}(Y\leq j)\right|\,.
$$
Since total variation distance is stronger than Kolmogorov distance, Corollary \ref{cor:ifr} also bounds the Kolmogorov distance between $W$ and our geometric distribution, and thus Obretenov's bound may be compared with ours.

Obretenov \cite{o81} shows that if $W$ is a nonnegative, integer-valued, IFR random variable with $\mathbb{E}W=\mu$ then
\begin{equation}\label{eq:o}
d_K\left(\mathcal{L}(W),\mbox{Geom}\left((1+\mu)^{-1}\right)\right)\leq\frac{\mu}{1+\mu}\left(1-\frac{\mbox{Var}(W)}{\mu(1+\mu)}\right)\,.
\end{equation}
Note that Obretenov chooses a geometric distribution having the same expectation as $W$, while we have chosen ours to have the same probability of being zero.

\subsection{Application to the P\'olya distribution}

Suppose $m$ balls are distributed randomly among $d\geq2$ urns, in such a way that all assignments are equally likely.  Let $W$ count the number of balls in the first urn.  Then $W\sim\mbox{Pya}(m,d)$ has a P\'olya distribution, with
$$
\mathbb{P}(W=k)=\frac{\binom{d+m-k-2}{m-k}}{\binom{d+m-1}{m}}\,,\qquad 0\leq k\leq m\,.
$$ 
It is straightforward to show that, with this definition 
$$
\mathbb{P}(W=k)^2\geq\mathbb{P}(W=k-1)\mathbb{P}(W=k+1)\\,
$$
for $1\leq k\leq m$.  Hence, $W$ is IFR. See, for example, page 177 of \cite{o81}.
We may thus apply Corollary \ref{cor:ifr}, which gives
\begin{equation}\label{eq:pold}
d_{TV}\left(\mathcal{L}(W),\mathcal{L}(Y)\right)\leq\frac{m}{d(d+m-1)}\,,
\end{equation}
where $Y\sim\mbox{Geom}\left(\frac{d-1}{d+m-1}\right)$.  In this case, Obretenov's bound (\ref{eq:o}) is
\begin{equation}\label{eq:polo}
d_{K}\left(\mathcal{L}(W),\mbox{Geom}\left(\frac{d}{d+m}\right)\right)\leq\frac{2m}{(d+1)(d+m)}\,.
\end{equation}
Our bound  is better than (\ref{eq:polo}) for $d$ large enough (specifically, $d^2+dm-3d-m>0$).  Note, however, that (\ref{eq:pold}) bounds the total variation distance, while (\ref{eq:polo}) bounds only the weaker Kolmogorov distance.  Our (\ref{eq:pold}) also improves upon bounds for geometric approximation of the P\'olya distribution in Example 3.1 of \cite{d10} and   Section 4 of \cite{pw00}.

A simple lower bound corresponding to (\ref{eq:pold}) is given by
\begin{equation}\label{eq:poll}
d_{TV}(\mathcal{L}(W),\mathcal{L}(Y))\geq|\mathbb{P}(W=1)-\mathbb{P}(Y=1)|=\frac{m(d-1)}{(d+m-2)(d+m-1)^2}\,.
\end{equation} 
In the case where $d$ is of order $O(m)$, this lower bound is of the same order as each of the upper bounds (\ref{eq:pold}) and (\ref{eq:polo}).  Some numerical comparison of the bounds (\ref{eq:pold})--(\ref{eq:poll}) is given in Table \ref{tab:polya}.
\begin{table}
\caption{Geometric approximation for $W\sim\mbox{Pya}(m,d)$.  Values of $p=(d-1)(d+m-1)^{-1}$ given to 4 d$.$p$.$; total variation distance between $W$ and $Y\sim\mbox{Geom}(p)$, and bounds (\ref{eq:pold})--(\ref{eq:poll}) given to 4 s$.$f.} 
\label{tab:polya}
\begin{center}
\begin{tabular}{ccccccc}
\hline
$m$ & $d$ & $p$ & $d_{TV}(\mathcal{L}(W),\mathcal{L}(Y))$ & (\ref{eq:pold}) & (\ref{eq:polo}) & (\ref{eq:poll})\\
\hline
200 & 200 & 0.4987 & 0.0009458 & 0.002506 & 0.004975 & 0.0006281\\
200 & 10 & 0.0431 & 0.03055 & 0.09569 & 0.1732 & 0.0001981\\
10 & 10 & 0.4737 & 0.02255 & 0.05263 & 0.09091 & 0.01385\\
10 & 200 & 0.9522 & 0.0002190 & 0.0002392 & 0.0004738 & 0.0002190\\
\hline
\end{tabular}
\end{center}
\end{table}

\subsection{Application to Poisson processes}\label{subsec:pp}

Let $\{N(t):t\geq0\}$ be a homogeneous Poisson process of rate $\lambda$ and let $T$ be an IFR random variable independent of $\{N(t):t\geq0\}$.  By Corollary 5.2 of \cite{rsz05}, $N(T)$ is also IFR.  Since $\mathbb{P}(N(T)=0)=\mathbb{E}e^{-\lambda T}$, $\mathbb{E}N(T)=\lambda\mathbb{E}T$, and $\mbox{Var}(N(T))=\lambda\mathbb{E}T+\lambda^2\mbox{Var}(T)$, we have from our Corollary \ref{cor:ifr} that
\begin{equation}\label{eq:ppd}
d_{TV}\left(\mathcal{L}(N(T)),\mbox{Geom}\left(\mathbb{E}e^{-\lambda T}\right)\right)\leq 1-(\mathbb{E}e^{-\lambda T})\left(1+\lambda\mathbb{E}T\right)\,,
\end{equation}
while Obretenov's result (\ref{eq:o}) gives
\begin{equation}\label{eq:ppo}
d_{K}\left(\mathcal{L}(N(T)),\mbox{Geom}\left((1+\lambda\mathbb{E}T)^{-1}\right)\right)\leq\frac{\lambda\mathbb{E}T}{1+\lambda\mathbb{E}T}\left(1-\frac{\mathbb{E}T+\lambda\mbox{Var}(T)}{\mathbb{E}T(1+\lambda\mathbb{E}T)}\right)\,.
\end{equation}
To give an explicit example where we may compare these bounds, suppose that $T\sim\Gamma(\alpha,\beta)$ has a gamma distribution with density function $\phi(x)$ proportional to $x^{\alpha-1}e^{-\beta x}$ for some $\alpha>1$ and $\beta>0$.  Then $T$ is IFR. 
Since
$$
\mathbb{E}e^{-\lambda T}=\left(1+\frac{\lambda}{\beta}\right)^{-\alpha}\,,\qquad\mathbb{E}T=\frac{\alpha}{\beta}\,,\qquad\mbox{Var}(T)=\frac{\alpha}{\beta^2}\,,
$$
the bounds (\ref{eq:ppd}) and (\ref{eq:ppo}) become, respectively,
\begin{equation}\label{eq:gammad}
d_{TV}\left(\mathcal{L}(N(T)),\mbox{Geom}\left(\left(1+\frac{\lambda}{\beta}\right)^{-\alpha}\right)\right)\leq 1-\left(1+\frac{\lambda}{\beta}\right)^{-\alpha}\left(1+\frac{\alpha\lambda}{\beta}\right)\,,
\end{equation}
and
\begin{equation}\label{eq:gammao}
d_{K}\left(\mathcal{L}(N(T)),\mbox{Geom}\left(\frac{\beta}{\beta+\alpha\lambda}\right)\right)\leq\frac{\alpha(\alpha-1)\lambda^2}{(\beta+\alpha\lambda)^2}\,.
\end{equation}

To compare (\ref{eq:gammad}) and (\ref{eq:gammao}), we use Taylor's theorem to note that for small $\lambda$ the upper bound of (\ref{eq:gammad}) is approximately equal to
$\alpha(\alpha-1)\lambda^2\left(2\beta^2\right)^{-1}$,
which is smaller than the upper bound of (\ref{eq:gammao}) whenever $(\sqrt{2}-1)\beta>\alpha\lambda$.  Finally, we again emphasise that (\ref{eq:gammad}) bounds the total variation distance, while (\ref{eq:gammao}) bounds only the weaker Kolmogorov distance.

We return to further applications of our results to Poisson processes in Section \ref{sec:hr}. 

\section{Approximation for DFR distributions}\label{sec:dfr}

In this section we present some further applications of our main result, Theorem \ref{thm:main}.  We will consider random variables which have the decreasing failure rate (DFR) property, so that the lower bound $\delta$ may be taken to be $\lim_{j\rightarrow\infty}r_W(j)$.  The applications we will consider will be to the number of customers served in a busy period of an M/G/1 queue, and to the time to extinction of a discrete birth-death process.  In each case we will construct the relevant random variable $W$ as the time at which a particular Markov chain on $\mathbb{Z}^+$ first visits the origin.  In this case, Shanthikumar \cite{s88} gives sufficient conditions for the DFR property to hold and an expression for the failure rate which will allow us to apply our Theorem \ref{thm:main}.

Let $\{Z_n:n\geq-1\}$ be a discrete-time Markov chain with state space $\mathbb{Z}^+$ and transition matrix $P=(p_{ij})$.  Let the entries of the matrix $P^+=(p^+_{ij})$ be given by $p^+_{ij}=\sum_{k=j}^\infty p_{ik}$ for $i,j\in\mathbb{Z}^+$.  Assume that the Markov chain starts at $Z_{-1}=1$ and define the hitting time
\begin{equation}\label{eq:hit}
W=\min\{n\geq0:Z_n=0\}\,.
\end{equation}
Without loss of generality in what follows, we may assume that the state $0$ is absorbing.  We have chosen to start our Markov chain at time $-1$ so that the support of $W$ matches that of our compound geometric distributions.

We say that the matrix $P^+$ is $\mbox{TP}_2$ if $p^+_{ik}p^+_{jl}\geq p^+_{il}p^+_{jk}$ for all $i<j$ and $k<l$.  Theorem 3.1 of \cite{s88} states that if $P^+$ is $\mbox{TP}_2$ then $W$ is DFR.  From the proof of that theorem, we also have that for such DFR hitting times $W$, $\widetilde{r}_W(j)\geq\widetilde{\delta}$ for all $j\in\mathbb{Z}^+$, where
\begin{equation}\label{eq:delta}
\widetilde{\delta}=\sum_{i=1}^\infty p_{i0}\lim_{j\rightarrow\infty}\mathbb{P}(Z_j=i|Z_j\geq1)\,.
\end{equation}
In order to evaluate this expression, we will therefore need an expression for the limiting distribution of our Markov chain conditional on non-absorption.

To translate a lower bound on $\widetilde{r}_W(j)$ into a lower bound on $r_W(j)$, we will use the following lemma, whose proof is straightforward and therefore omitted.
\begin{lem}\label{lem:delta}
Let $W$ be a nonnegative, integer-valued random variable with $\widetilde{r}_W(j)\geq\widetilde{\delta}$ for all $j\in\mathbb{Z}^+$.  Then $r_W(j)\geq\widetilde{\delta}\left(1-\widetilde{\delta}\right)^{-1}$ for all $j\in\mathbb{Z}^+$. 
\end{lem}

\subsection{Customers served during a busy period of an M/G/1 queue}

Consider the M/G/1 queue of Section \ref{sec:queue}, with customers arriving at rate $\lambda$ and i$.$i$.$d$.$ service times with the same distribution as $S$.  Again letting $\rho=\lambda\mathbb{E}S$, we will assume throughout that $\rho<1$.  We will also assume that $S$ is IFR, so that, by Shanthikumar's \cite{s88} Theorem 5.1, the number of customers served during a busy period is DFR. 

Consider the embedded Markov chain $\{Z_n:n\geq-1\}$, where $Z_{-1}=1$ and $Z_n$ represents the number of customers in the system after the departure of customer $n$ (with customers labelled $0,1,2,\ldots$).  Then $W+1$, with the hitting time $W$ given by (\ref{eq:hit}), is the number of customers served during a busy period of the queue.

This Markov chain has the transition probabilities $p_{00}=1$ and $p_{ij}=g(j+1-i)$, where
\begin{displaymath}
g(k)=\left\{ \begin{array}{cl}
\frac{1}{k!}\mathbb{E}\left[e^{-\lambda S}(\lambda S)^k\right] & \textrm{if } k\geq0\,,\\
0 & \textrm{if }k<0\,.
\end{array} \right.
\end{displaymath}
Hence, 
\begin{equation}\label{eq:p}
p=\mathbb{P}(W=0)=p_{10}=\mathbb{E}e^{-\lambda S}\,.
\end{equation}
We also have that $\mathbb{E}W=\rho(1-\rho)^{-1}$.  See, for example, page 217  of \cite{k75}.

Since $p_{i0}=0$ for $i>1$, the lower bound $\widetilde{\delta}$ given by (\ref{eq:delta}) becomes $\widetilde{\delta}=p\theta$, where $\theta=\lim_{j\rightarrow\infty}\mathbb{P}(Z_j=1|Z_j\geq1)$.  To find an expression for $\theta$ we use a formula due to Kyprianou \cite{k72}.  Suppose the density of the service time $S$ has Laplace transform $\varphi$, and let $\xi$ be the real solution of $1+\lambda\varphi^\prime(s)=0$ nearest the origin.  By a result on page 829 of \cite{k72}, we then have that
\begin{equation}\label{eq:theta}
\theta=\frac{\xi-\lambda+\lambda\varphi(\xi)}{(\xi-\lambda)\varphi(\lambda)}\,.
\end{equation}
Using Lemma \ref{lem:delta}, we may then take the lower bound $\delta=p\theta(1-p\theta)^{-1}$ in Theorem \ref{thm:main} and we obtain the following.
\begin{thm}\label{thm:queue}
Let $W+1$ be the number of customers served in a busy period of an M/G/1 queue with arrival rate $\lambda$ and service time $S$.  Let $p$ and $\theta$ be given by (\ref{eq:p}) and (\ref{eq:theta}), respectively.  Suppose that $S$ is IFR and that $\rho=\lambda\mathbb{E}S<1$.  Let $N\sim\mbox{Geom}(p)$ and $Y=\sum_{i=1}^NX_i$, where $X,X_1,X_2,\ldots$ are i$.$i$.$d$.$ with $(1-p)\theta\mathbb{E}X\geq1-p\theta$.  Then
$$
d_{TV}(\mathcal{L}(W),\mathcal{L}(Y))\leq H_p(X)\left(\frac{(1-p)\mathbb{E}X}{p}-\frac{\rho}{1-\rho}\right)\,.
$$
\end{thm}
The number of customers served in a busy period of this queueing system is closely related to the total progeny of a certain branching process, and so our Theorem \ref{thm:queue} may also be applied in that setting.  If we define the offspring of a customer to be the other customers who arrive while he is being served, the number of customers served during a busy period has the same distribution as the total progeny of the customer initiating the busy period.  See page 284 of \cite{a03} for further details.

To illustrate our Theorem \ref{thm:queue}, we consider the example where $S\sim\Gamma(k,\beta)$ has an Erlang distribution for some integer $k\geq1$ and some $\beta>0$.  In this case, $S$ is indeed IFR.  Since $\mathbb{E}S=k\beta^{-1}$, our condition on $\rho$ requires that $k\lambda<\beta$.  

Using (\ref{eq:p}), $p=\left(1+\lambda\beta^{-1}\right)^{-k}$.  The Erlang density has Laplace transform $\varphi(s)=\beta^k(s+\beta)^{-k}$, and since $k\lambda<\beta$ it is straightforward to check that $\xi=(\lambda k\beta^k)^{1/(k+1)}-\beta$ and that therefore, by (\ref{eq:theta}),
$$
\theta=\left(\frac{\beta+\lambda}{\beta}\right)^k\left(1-\frac{1}{A}\right)\,,\mbox{ where }A=(\beta+\lambda)\left(\frac{k^k}{\lambda\beta^k}\right)^{1/(k+1)}-k\,.
$$
Theorem \ref{thm:queue} thus requires that we choose
\begin{equation}\label{eq:xq}
\mathbb{E}X\geq\frac{\beta^k}{(A-1)((\beta+\lambda)^k-\beta^k)}\,.
\end{equation}
If we choose $X$ such that equality holds in (\ref{eq:xq}), the upper bound of Theorem \ref{thm:queue} becomes $H_p(X)U\leq U$, where
\begin{equation}\label{eq:u}
U=\frac{1}{A-1}-\frac{k\lambda}{\beta-k\lambda}\,.
\end{equation}
Some numerical illustration of this bound is given in Table \ref{tab:erlang}.
\begin{table}
\caption{Some values of the upper bound $U$ of (\ref{eq:u}) for the Erlang service time example.  Invalid parameter choices are marked with --.}
\label{tab:erlang}
\begin{center}
\begin{tabular}{ccccccc}
\hline
\multirow{2}{*}{$k$}& \multirow{2}{*}{$\lambda$} & \multicolumn{5}{c}{$\beta$}\\
\cline{3-7}
& & 0.1 & 0.5 & 1 & 1.5 & 10 \\
\hline
\multirow{5}{*}{1} & 0.001 & 0.1134 & 0.0470 & 0.0327 & 0.0265 & 0.0101 \\
& 0.005 & 0.3183 & 0.1134 & 0.0769 & 0.0617 & 0.0229 \\
& 0.01 & 0.5652 & 0.1714 & 0.1134 & 0.0901 & 0.0327 \\
& 0.05 & $>1$ & 0.5652 & 0.3183 & 0.2388 & 0.0769 \\
& 0.1 & -- & $>1$ & 0.5652 & 0.3978 & 0.1134 \\
\hline
\multirow{5}{*}{5} & 0.001 & 0.3784 & 0.1985 & 0.1588 & 0.1406 & 0.0846 \\
& 0.005 & $>1$ & 0.3783 & 0.2781 & 0.2378 & 0.1294 \\
& 0.01 & $>1$ & 0.5619 & 0.3784 & 0.3137 & 0.1588 \\
& 0.05 & -- & $>1$ & $>1$ & 0.8366 & 0.2781 \\
& 0.1 & -- & -- & $>1$ & $>1$ & 0.3784 \\
\hline
\multirow{5}{*}{10} & 0.001 & 0.5777 & 0.2827 & 0.2271 & 0.2025 & 0.1282 \\
& 0.005 & $>1$ & 0.5777 & 0.4027 & 0.3402 & 0.1874 \\
& 0.01 & -- & 0.9890 & 0.5777 & 0.4614 & 0.2272 \\
& 0.05 & -- & -- & $>1$ & $>1$ & 0.4027 \\
& 0.1 & -- & -- & -- & $>1$ & 0.5777 \\
\hline
\end{tabular}
\end{center}
\end{table}

\subsection{Time to extinction of a birth-death process}

We let  $\{Z_n:n\geq-1\}$ be the Markov chain with $Z_{-1}=1$, $p_{00}=1$, and 
\begin{displaymath}
p_{ij}=\left\{ \begin{array}{cl}
p_i & \textrm{if } j=i+1\,,\\
q_i & \textrm{if }j=i-1\,,\\
r_i & \textrm{if }j=i\,,\\
0 & \textrm{otherwise}\,,
\end{array} \right.
\end{displaymath}
for $i\geq1$.  Let $W$ be the hitting time defined by (\ref{eq:hit}): the time when this discrete birth-death process becomes extinct.

Clearly we have that $p=\mathbb{P}(W=0)=q_1$ and, from (\ref{eq:delta}),
\begin{equation}\label{eq:bddelta}
\widetilde{\delta}=q_1\lim_{j\rightarrow\infty}\mathbb{P}(Z_j=1|Z_j\geq1)\,.
\end{equation} 
To find an expression for this limit, we use the famous Karlin--McGregor \cite{km59} representation of the $n$-step transition probabilities of this chain.  Define $\pi_1=1$ and
$$
\pi_j=\frac{p_1\cdot p_2\cdots p_{j-1}}{q_2\cdot q_3\cdots q_j}\,,
$$  
for $j\geq2$.  Then Karlin and McGregor \cite{km59} show that there is a unique positive measure $\psi$, of total mass 1, supported on $[-1,1]$ such that
$$
p_{ij}(n)=\mathbb{P}(Z_n=j|Z_0=i)=\pi_j\int_{-1}^1x^nQ_i(x)Q_j(x)\,d\psi(x)\,,\qquad\mbox{for }i,j\geq1\,,
$$
where $\{Q_j:j\geq1\}$ is a sequence of polynomials (orthogonal with respect to $\psi$) satisfying the relations $Q_1(x)=1$, $p_1Q_2(x)=x-r_1$, and 
$$
xQ_j(x)=q_jQ_{j-1}(x)+r_jQ_j(x)+p_jQ_{j+1}(x)\,,
$$
for $j\geq2$.  Following the notation of van Doorn and Schrijner \cite{vds95}, $Q_{j+1}$ has $j$ distinct zeros which we denote $x_{1j}<x_{2j}<\cdots<x_{jj}$.  We write $\eta=\lim_{k\rightarrow\infty}x_{kk}$ and
$$
C_n(\psi)=\frac{\int_{-1}^0(-x)^n\,d\psi(x)}{\int_0^1x^n\,d\psi(x)}\,.
$$
In what follows we make the following assumptions:
\begin{eqnarray}
\label{eq:bd1}\sum_{k=1}^\infty(p_k\pi_k)^{-1}&=&\infty\,,\\
\label{eq:bd2}\eta&<&1\,,\\
\label{eq:bd3}\lim_{n\rightarrow\infty}C_n(\psi)&=&0\,,\\
\label{eq:bd4}r_j&\geq&\frac12\mbox{ for all }j\geq1\,.
\end{eqnarray}
Assumption (\ref{eq:bd1}) guarantees that the birth-death process does eventually reach extinction: see Section 4 of \cite{vds95}.  Assumptions (\ref{eq:bd2}) and (\ref{eq:bd3}) are used to ensure that the limit (\ref{eq:bddelta}) exists, and are taken from Lemma 4.1 of \cite{vds95}.  Finally, assumption (\ref{eq:bd4}) is sufficient to guarantee that the transition matrix of our birth-death chain is $\mbox{TP}_2$, and hence that the extinction time $W$ is DFR. See page 6 of \cite{fj11} and Remark 3.2 of \cite{s88}.

We note that Section 3 of \cite{vds95} gives several conditions under which the assumption (\ref{eq:bd3}) holds and which may be used to check its validity in practice.

Under the assumptions (\ref{eq:bd1})--(\ref{eq:bd4}), Lemma 4.1 of \cite{vds95} gives us that $\widetilde{\delta}=1-\eta$, and so (by Lemma \ref{lem:delta}) we may take $\delta=\eta^{-1}(1-\eta)$ in Theorem \ref{thm:main}.  Applying that result, we then obtain the following.
\begin{thm}
Let $W$ be the time to extinction of the discrete birth-death process defined above.  Assume (\ref{eq:bd1})--(\ref{eq:bd4}) hold.  Let $N\sim\mbox{Geom}(q_1)$ and $Y=\sum_{i=1}^NX_i$, where $X,X_1,X_2,\ldots$ are i$.$i$.$d$.$ with $(1-q_1)(1-\eta)\mathbb{E}X\geq q_1\eta$.  Then
$$
d_{TV}(\mathcal{L}(W),\mathcal{L}(Y))\leq H_{q_1}(X)\left(\frac{(1-q_1)\mathbb{E}X}{q_1}-\mathbb{E}W\right)\,.
$$
\end{thm}  

Finally, note that Brown \cite{b15} considers exponential approximation for hitting times of birth-death processes in continuous time, taking advantage of monotonicity of the failure rate in his work.  See also the references within Brown's work.

\section{Proof of Theorem \ref{thm:main}}\label{sec:proof}

We use this section to give the proof of our main result, Theorem \ref{thm:main}.  The proof is based on Stein's method for compound geometric approximation.  Stein's method was first applied to the problem of approximation by a geometric distribution by Barbour and Gr\"ubel \cite{bg95} and Pek\"oz \cite{p96}.  More recent developments in Stein's method for geometric approximation are given in \cite{pw00} and \cite{prr13}.  Stein's method has previously been used in the compound geometric case in \cite{d10}, and compound geometric distributions have appeared in conjunction with Stein's method in papers by Bon \cite{b06}, Pek\"oz and R\"ollin \cite{pr11}, and Pek\"oz, R\"ollin and Ross \cite{prr13}.  The interested reader is also referred to \cite{bc05} and references therein for an introduction to Stein's method more generally.

Throughout this section we will let $W$ and $Y$ be as defined in Theorem \ref{thm:main}, and $H_p(X)$ be given by (\ref{eq:steinfactor}).  We define the random variable $V$ to be such that $V+X=_{st}(W|W>0)$, where $=_{st}$ denotes equality in distribution.

We will employ the usual stochastic ordering in what follows.  Recall that for any two random variables $T$ and $U$, $T$ is said to be stochastically smaller than $U$ (written $T\leq_{st}U$) if $\mathbb{P}(T>j)\leq\mathbb{P}(U>j)$ for all $j$.
\begin{lem}\label{lem:ord1}
Let $W$ be a nonnegative, integer-valued random variable with $\mathbb{P}(W=0)=p$ and $r_W(j)\geq\delta>0$ for all $j\in\mathbb{Z}^+$.  Let $V$ be as above and suppose that 
$\mathbb{E}X\geq\frac{p}{(1-p)\delta}$.  Then $V+X\leq_{st}W+X$.
\end{lem}  
\begin{proof}
From the definition of $V$, the required stochastic ordering will follow if 
\begin{equation}\label{eq:ord1}
(1-p)\mathbb{P}(W+X>j)\geq\mathbb{P}(W>j)\,,
\end{equation}
for all $j\in\mathbb{Z}^+$.

Conditioning on $X$ (which is independent of $W$), we write
$$
\mathbb{P}(W+X>j)=\mathbb{P}(W>j)+\mathbb{E}\left[\sum_{k=j+1-X}^j\mathbb{P}(W=k)\right]\,.
$$
Using this, we rearrange (\ref{eq:ord1}) to obtain that the required stochastic ordering holds if
\begin{equation}\label{eq:ord2}
\frac{1}{\mathbb{P}(W>j)}\mathbb{E}\left[\sum_{k=j+1-X}^j\mathbb{P}(W=k)\right]\geq\frac{p}{1-p}\,.
\end{equation}
Now, if $r_W(k)\geq\delta$ for all $k$ then
$$
\frac{1}{\mathbb{P}(W>j)}\mathbb{E}\left[\sum_{k=j+1-X}^j\mathbb{P}(W=k)\right]\geq\mathbb{E}\left[\sum_{k=j+1-X}^jr_W(k)\right]\geq\delta\mathbb{E}X\,.
$$
Hence, if $\mathbb{E}X\geq\frac{p}{(1-p)\delta}$ then (\ref{eq:ord2}) holds and our lemma follows.
\end{proof}

The proof of Theorem \ref{thm:main} then goes along similar lines to that of Proposition 3.1 in \cite{d10}.  For $A\subseteq\mathbb{Z}^+$ we let $f_A:\mathbb{Z}^+\mapsto\mathbb{R}$ be such that $f_A(0)=0$ and
\begin{equation}\label{eq:stein}
I(j\in A)-\mathbb{P}(Y\in A)=(1-p)\mathbb{E}f_A(j+X)-f_A(j)\,,
\end{equation}  
$I(\cdot)$ denoting an indicator function.  We then note the following property of $f_A$.
\begin{lem}\label{lem:stein}
Let $f_A$ be as above.  Then $\sup_{j\in\mathbb{Z}^+}\left|f_A(j+1)-f_A(j)\right|\leq\frac{1}{p}H_p(X)$. 
\end{lem}  
\begin{proof}
From the definition (\ref{eq:stein}), it is easy to check that 
\begin{equation}\label{eq:proof0}
f_A(j)=-\sum_{i=0}^\infty(1-p)^i\left[\mathbb{P}(j+X_1+\cdots+X_i\in A)-\mathbb{P}(Y+X_1+\cdots+X_i\in A)\right]\,,
\end{equation}
from which it follows that $|f_A(j+1)-f_A(j)|$ may be bounded by
\begin{equation}\label{eq:proof1}
\sum_{i=0}^\infty(1-p)^i\left|\mathbb{P}(j+1+X_1+\cdots+X_i\in A)-\mathbb{P}(j+X_1+\cdots+X_i\in A)\right|\,.
\end{equation}
To complete the proof, we bound (\ref{eq:proof1}) in two different ways.  Firstly, letting $N\sim\mbox{Geom}(p)$, we write this as 
\begin{multline}\label{eq:proof2}
\frac{1}{p}|\mathbb{P}(j+1+X_1+\cdots+X_N\in A)-\mathbb{P}(j+X_1+\cdots+X_N\in A)|\\
\leq\frac{1}{p}d_{TV}(\mathcal{L}(Y),\mathcal{L}(Y+1))\leq1+\frac{(1-p)}{p}\mathbb{P}(X>1)\,,
\end{multline}
where the final inequality uses Theorem 3.1 of \cite{vc96}.  Alternatively, we have that
\begin{multline*}
\left|\mathbb{P}(j+1+X_1+\cdots+X_i\in A)-\mathbb{P}(j+X_1+\cdots+X_i\in A)\right|\\
\leq d_{TV}(\mathcal{L}(X_1+\cdots+X_i),\mathcal{L}(X_1+\cdots+X_i+1))\,.
\end{multline*}
We may then follow the analysis of Theorem 3.1 of \cite{prr13} to obtain
$$
|f_A(j+1)-f_A(j)|\leq1+\sqrt{\frac{-2}{u\log(1-p)}}\,,
$$
where $u=1-d_{TV}(\mathcal{L}(X),\mathcal{L}(X+1))$.  This completes the proof.
\end{proof}
\begin{remark}
\emph{Lemma \ref{lem:stein} improves upon part of Theorem 2.1 of \cite{d10}, by presenting a sharper bound and removing a restriction on the support of $X$.  This may, in turn, be used to improve on Proposition 3.1 of \cite{d10}.  For a general $X$, Lemma \ref{lem:stein} gives the bound $|f_A(j+1)-f_A(j)|\leq p^{-1}$ (which is the same bound given in \cite{d10}), but also shows that a better bound is possible when $\mathbb{P}(X=1)$ is large (informally, when $Y$ is close to a geometric distribution) or when $X$ is smooth (in the sense that the total variation distance between $X$ and $X+1$ is small).  Note that the bound $|f_A(j+1)-f_A(j)|\leq p^{-1}$ is the best possible without imposing restrictions on $X$.  To see this, suppose that $X=2$ a$.$s$.$ and let $A=2\mathbb{Z}^+$.  In this case, (\ref{eq:proof0}) gives $f_A(j+1)-f_A(j)=(-1)^jp^{-1}$.  The inequalities (\ref{eq:proof2}) are also sharp, in that if $Y\sim\mbox{Geom}(p)$ then $d_{TV}(\mathcal{L}(Y),\mathcal{L}(Y+1))=p$.}
\end{remark}
Using the definitions of $f_A$ and $V$, we may write 
\begin{multline}\label{eq:ord3}
\mathbb{P}(W\in A)-\mathbb{P}(Y\in A)=(1-p)\mathbb{E}\left[f_A(W+X)-f_A(V+X)\right]\\
=(1-p)\sum_{j=0}^\infty\left[f_A(j+1)-f_A(j)\right]\left[\mathbb{P}(W+X>j)-\mathbb{P}(V+X>j)\right]\,.
\end{multline}
Now, under the conditions of Theorem \ref{thm:main}, Lemma \ref{lem:ord1} gives us that $V+X\leq_{st}W+X$.  Hence, bounding (\ref{eq:ord3}) using Lemma \ref{lem:stein} gives
$$
|\mathbb{P}(W\in A)-\mathbb{P}(Y\in A)|\leq\frac{(1-p)}{p}H_p(X)\mathbb{E}\left[W-V\right]=H_p(X)\mathbb{E}\left[Y-W\right]\,,
$$
where the final equality follows from the definition of $V$.  We have thus established Theorem \ref{thm:main}.
\begin{remark}
\emph{The techniques of this section may also be used to bound the Wasserstein distance $d_W(\mathcal{L}(W),\mathcal{L}(Y))=\sup_h\left|\mathbb{E}h(W)-\mathbb{E}h(Y)\right|$, where the supremum is taken over all 1-Lipschitz functions $h$.  Under the conditions of Theorem \ref{thm:main}, we follow the above methods to obtain the bound $d_W(\mathcal{L}(W),\mathcal{L}(Y))\leq\mathbb{E}Y-\mathbb{E}W$.  This bound is sharp.  Suppose, for example, that $W$ is IFR and $Y\sim\mbox{Geom}(p)$.  Then $W\leq_{st}Y$ (\cite[Theorem 1.B.1]{ss07}) and so $d_W(\mathcal{L}(W),\mathcal{L}(Y))=\mathbb{E}Y-\mathbb{E}W$.  See Theorem 1.A.11 of \cite{ss07}.}
\end{remark}

\section{Some further results}\label{sec:hr}

In this section we note two results, closely related to Theorem \ref{thm:main}, which may prove useful in applications.  Potential applications are indicated for each via short examples.

\subsection{Approximation for translated distributions}

Let $W$ be as in Theorem \ref{thm:main}, and let $W_m=_{st}(W-m|W\geq m)$ for some $m\in\mathbb{Z}^+$.  In many cases it is more natural to seek a compound geometric approximation for $W_m$ (for some $m\geq1$) than for $W$.  This may be achieved in a straightforward way using Theorem \ref{thm:main}.  We note that the failure rate of $W_m$ may be bounded from below by
\begin{equation}\label{eq:deltam}
\delta_m=\min_{j\geq m}r_W(j)\geq\delta\,,
\end{equation}
and that if $W$ has monotone failure rate, $W_m$ inherits this property.  Letting
\begin{equation}\label{eq:pm}
p_m=\mathbb{P}(W_m=0)=\frac{\mathbb{P}(W=m)}{\mathbb{P}(W\geq m)}\,,
\end{equation}
we may apply Theorem \ref{thm:main} to $W_m$ to obtain the following corollary.
\begin{cor}\label{cor:shift}
Let $W$ be a nonnegative, integer-valued random variable, $m\in\mathbb{Z}^+$, and $W_m=_{st}(W-m|W\geq m)$.  Let $\delta_m$ and $p_m$ be given by (\ref{eq:deltam}) and (\ref{eq:pm}), respectively.  Let $Y=\sum_{i=1}^NX_i$, where $N\sim\mbox{Geom}(p_m)$ and $X,X_1,X_2,\ldots$ are i$.$i$.$d$.$ with
$$
\mathbb{E}X\geq\frac{p_m}{(1-p_m)\delta_m}\,.
$$
Then
$$
d_{TV}(\mathcal{L}(W_m),\mathcal{L}(Y))\leq H_{p_m}(X)\left(\mathbb{E}Y-\mathbb{E}W_m\right)\,.
$$
\end{cor}
To illustrate one situation in which such a result would be useful, consider the Markov chain $\{Z_n:n\geq-1\}$ with state space $\{0,1,2\}$, $Z_{-1}=2$, and transition matrix
\[ \left( \begin{array}{ccc}
1 & 0 & 0 \\
\alpha_1 & \beta_1 & \epsilon_1 \\
\alpha_2 & \beta_2 & \epsilon_2 \end{array} \right)\,,\]
where, for $i\in\{1,2\}$, we have that $\alpha_i,\beta_i,\epsilon_i\in(0,1)$ with $\alpha_i+\beta_i+\epsilon_i=1$ and where we consider $\epsilon_i$ to be small.

Letting $W$ be the hitting time $W=\min\{n\geq0:Z_n=0\}$, the most natural geometric-type approximation in this setting is to approximate $W_1$ by a geometric distribution with parameter close to $\alpha_1$. We will show that this is easily achieved using Corollary \ref{cor:shift}.  Elementary calculations show that, since $\mathbb{P}(W=0)=\alpha_2$,
$$
p_1=\frac{\alpha_1\beta_2+\alpha_2\epsilon_2}{1-\alpha_2}\,,\;\;\mathbb{E}W_1=\frac{\mathbb{E}W}{1-\alpha_2}-1\,,\;\mbox{and}\;\;\mathbb{E}W=\frac{(1-\beta_1)\epsilon_2+\beta_2(1+\epsilon_1)}{(1-\beta_1)(1-\epsilon_2)-\beta_2\epsilon_1}\,.
$$

For simplicity in what follows, we will assume that $\alpha_1\geq\alpha_2$ and that $\beta_1(\beta_2+\epsilon_2)\geq\beta_2(\beta_1+\epsilon_1)$.  These conditions are sufficient to guarantee that $W$ is IFR (using Theorem 4.1 of \cite{rsz05}).  In this case we may take $X=1$ a$.$s$.$ in Corollary \ref{cor:shift}, as we did in the IFR examples we have previously considered.

Corollary \ref{cor:shift} then gives us the bound
$d_{TV}(\mathcal{L}(W_1),\mbox{Geom}(p_1))\leq A+B+C$, where
$$
A=\frac{\alpha_1\beta_2+\alpha_2\epsilon_2}{\alpha_2(1-\alpha_2)+(\alpha_1-\alpha_2)(1-\alpha_2-\epsilon_2)}\,,\;\;
B=\frac{(\alpha_1\beta_2+\alpha_2\epsilon_2)(1+\epsilon_1-\epsilon_2)}{(1-\alpha_2)(\alpha_1-\epsilon_1(\alpha_1+\alpha_2))}\,,
$$ 
and
$$
C=\frac{\epsilon_2(\alpha_1-\alpha_2)(\alpha_1\beta_2+\alpha_2\epsilon_2)}{(1-\alpha_2)^2(\alpha_2\epsilon_1+\alpha_1(1-\epsilon_2))}\,.
$$
We conclude this illustration by noting that if either $\epsilon_1=\epsilon_2=0$ or $\alpha_1=\alpha_2$ then our upper bound is zero, as expected.

\subsection{Hazard rate ordering}

In this section we will need the hazard rate ordering.  For two nonnegative random variables $T$ and $U$, $T$ is said to be smaller than $U$ in the hazard rate order (denoted $T\leq_{hr}U$) if $r_T(j)\geq r_U(j)$ for all $j$.  See, for example, Section 1.B of \cite{ss07}.

In proving Theorem \ref{thm:main}, Lemma \ref{lem:ord1} gave conditions under which $V+X\leq_{st}W+X$, which then allowed us to deduce a compound geometric approximation bound.  In the case of geometric approximation, we use the hazard rate order to express conditions under which this stochastic ordering holds.

If $X=1$ a$.$s$.$ (so we are in the geometric approximation case, and $H_p(X)=p$), then (\ref{eq:ord2}) tells us that $V+1\leq_{st}W+1$ if 
$$
r_W(j)\geq\frac{p}{1-p}=r_N(j)\,,
$$
where $p=\mathbb{P}(W=0)$ and $N\sim\mbox{Geom}(p)$.  That is, if $W\leq_{hr}N$ then $V+1\leq_{st}W+1$ and the bound of Theorem \ref{thm:main} holds with $X=1$ almost surely.  A similar argument shows that if $N\leq_{hr}W$ then $W+1\leq_{st}V+1$ and we obtain an analogous geometric approximation result.  In fact, we have the following.
\begin{thm}\label{cor:hr}
Let $W$ be a nonnegative, integer-valued random variable with $p=\mathbb{P}(W=0)$.  Let $N\sim\mbox{Geom}(p)$ and suppose that either $W\leq_{hr}N$ or $N\leq_{hr}W$.  Then
$$d_{TV}(\mathcal{L}(W),\mathcal{L}(N))\leq\left|1-p(1+\mathbb{E}W)\right|\,.$$
\end{thm}  
To illustrate this result, we return to the setting of Section \ref{subsec:pp}, and let $\{N(t):t\geq0\}$ be a Poisson process of rate $\lambda$ and $T$ be a nonnegative random variable independent of $\{N(t):t\geq0\}$.  We have the following.
\begin{cor}
Let $\{N(t):t\geq0\}$ and $T$ be as above.  Let $p=\mathbb{E}e^{-\lambda T}$ and $\mu=\lambda p(1-p)^{-1}$.  Let $\eta\sim\mbox{Exp}(\mu)$ have an exponential distribution with mean $\mu^{-1}$ and suppose that either $T\leq_{hr}\eta$ or $\eta\leq_{hr}T$.  Then
$d_{TV}(\mathcal{L}(N(T)),\mbox{Geom}(p))\leq\lambda p\left|\mu^{-1}-\mathbb{E}T\right|$.
\end{cor}
\begin{proof}
We note that $p=\mathbb{P}(N(T)=0)$, $\mathbb{E}N(T)=\lambda\mathbb{E}T$, and that $N(\eta)\sim\mbox{Geom}(p)$.  The bound follows from Theorem \ref{cor:hr} if either $N(T)\leq_{hr}N(\eta)$ or $N(\eta)\leq_{hr}N(T)$.

Consider first the inequality $N(T)\leq_{hr}N(\eta)$.  Using Theorem 1.B.14 of \cite{ss07}, this holds if $T\leq_{hr}\eta$.   To see this, we need to verify that if $Z_\alpha\sim\mbox{Po}(\alpha)$ has a Poisson distribution with mean $\alpha$ then $Z_\alpha\leq_{hr}Z_\beta$ whenever $\alpha\leq\beta$.  This is most easily checked by noting that $\mathbb{P}(Z_\beta=j)\mathbb{P}(Z_\alpha=j)^{-1}$ is increasing in $j$, and then using Theorem 1.C.1 of \cite{ss07} to get the required hazard rate ordering.

Similarly, if $\eta\leq_{hr}T$ then $N(\eta)\leq_{hr}N(T)$ and the stated upper bound holds. 
\end{proof}

\subsection*{Acknowledgements}
The author thanks Mark Brown for providing a preprint of his paper \cite{b15}, and an anonymous referee for pointing out useful material in the literature and suggestions which led to improvements in the results presented.

\bibliography{hr_paper_jap}

\end{document}